\documentclass[12pt]{amsart}

\usepackage{amssymb}

\usepackage{enumerate}
\usepackage{epstopdf}
\usepackage{color}
\usepackage{float}
\usepackage{amsthm}
\usepackage{amsmath}
\theoremstyle{plain}

\usepackage{graphicx}

\makeatletter
\@namedef{subjclassname@2010}{%
  \textup{2010} Mathematics Subject Classification}
\makeatother

\theoremstyle{plain}

\newtheorem{theorem}{Theorem}[section]
\newtheorem{corollary}[theorem]{Corollary}
\newtheorem{lemma}[theorem]{Lemma}
\newtheorem{proposition}[theorem]{Proposition}

\theoremstyle{definition}

\newtheorem{remark}[theorem]{Remark}
\newtheorem{example}[theorem]{Example}

\numberwithin{equation}{section}

\begin{document}


\baselineskip=17pt


\title[The critical orbit structure of quadratic polynomials in $\mathbb{Z}_p$]{The critical orbit structure of quadratic polynomials in $\mathbb{Z}_p$}

\author[C. Mullen]{Cara Mullen}
\address{Department of Mathematics, Statistics, and Computer Science\\ University of Illinois at Chicago\\ Chicago, IL 60607 USA}
\email{cmulle4@uic.edu}

\date{}

\begin{abstract}
We study the forward orbit of the critical point for polynomials of the form $f_c=z^2+c$ defined over $\mathbb{Z}_p$.  Hubbard trees capture the dynamical behavior for such maps with finite critical orbit in $\mathbb{C}$.  We suggest a notion of Hubbard trees in the non-Archimedean setting, and describe the possible structures that arise for polynomials in $\mathbb{Z}_p$.  As an example, we take a closer look at the dynamics of $f_c$ for $c\in\mathbb{Z}_3$.

\end{abstract}

\subjclass[2010]{Primary 11S82; Secondary 37P05}

\keywords{non-Archimedean dynamical systems}

\maketitle

\section{Introduction}

In one-dimensional complex dynamics, the forward orbit of the critical points completely determine the dynamical behavior of a polynomial $f$.  If all such orbits are finite, the polynomial is called \textit{post-critically finite} (PCF), and it has an associated Hubbard tree which can capture that behavior.  In particular, the Hubbard tree illustrates the \textit{orbit type} of a critical point $\alpha$, the minimal pair $(m,n)$ such that $f^{m+n}(\alpha)=f^m(\alpha)$, and the geometry of these orbits within the Julia set of $f$. Hubbard trees have been well studied, and a full classification of finite critical orbit trees is known in this  setting.  For more on complex Hubbard trees, see  \cite{Bruin}, \cite{Orsay}, or \cite{Poirier}.

The goal of this article is to begin to understand what the analogous objects are in a non-Archimedean setting. We explore the critical orbits and corresponding trees for quadratic polynomials of the form $f_c(z) = z^2+c$ defined over $\mathbb{Z}_p$.  For each $c$ in $\mathbb{Z}_p$, the orbit of the unique critical point is bounded, so $c$ lies in the $p$-adic analogue of the Mandelbrot set, $\mathcal{M}_p= \{ c\in\mathbb{C}_p : \mbox{ the critical orbit of } f_c(z) = z^2 + c \mbox{ is bounded}\}$, which is simply the closed unit disk.  $\mathcal{M}_p$ is more than just a simple disk, however; it has the structure of an infinitely branching tree with many interesting features, which lead to striking patterns in the dynamics of these uni-critical polynomials.

For example, for each prime $p>2$, it is known that if 0 has period $n$ under iteration of the reduced map $\widetilde{f_c(z)} = z^2+\tilde{c}$, then either 0 is periodic of exact period $n$ under iteration of $f$, or it has infinite orbit and is attracted to an attracting $n$-cycle.  This follows from the work of Jones \cite{Jones} and Rivera-Letelier \cite{Juan}, as well as Benedetto, Ingram, Jones and Levy \cite{Benedetto}.  See the discussion in Section \ref{proof} for details.

A natural extension of this work is to study what happens when 0 is strictly pre-periodic under iteration of $\widetilde{f_c(z)}=z^2 + \tilde{c}$.  Building on earlier results from Pezda \cite{Pezda} and Morton-Silverman \cite{Morton}, we prove the following theorem:

\begin{theorem}\label{preper}
Let $p\geq 3$ and consider the critical orbit for $f_c(z)=z^2+c$, $c\in\mathbb{Z}_p$.  If 0 is strictly pre-periodic with orbit type $(m,n)$ (mod $p$), $m>0$, then either 0 has orbit type $(m,n)$ over $\mathbb{Z}_p$ or there exists some $k\geq 1$ and $r|(p-1)$ in $\mathbb{Z}$ such that 
\begin{enumerate}
\item 0 has orbit type $(m,n)$ (mod $p^i$) for all $i\leq k$, and 
\item 0 has orbit type $(m, rn)$ (mod $p^j$) for all $j> k$, with $r=p$ if $p=3$.
\end{enumerate}
Otherwise, 0 has infinite orbit, with orbit type $(m, n_i)$ (mod $p^i$) for all $i\geq 1$, where $n_i$ is the length of the cycle on which 0 lands when its orbit is calculated (mod $p^i$).
\end{theorem}

\begin{remark}
For every prime $p$, there are only finitely many PCF polynomials of the form $z^2+c$, with $c$ in $\mathbb{Z}_p$, so there is a uniform bound on $k$ depending only on $p$.  See Section \ref{small} for further comments.
\end{remark}

\begin{remark}
In both the periodic and strictly pre-periodic work, the case of $p=2$ is treated separately.  A brief discussion of what happens in $\mathbb{Z}_2$ is included in Section \ref{p2}.
\end{remark}

With the $p$-adic ultrametric, we may visualize $\mathbb{Z}_p$ as the set of ends of a rooted tree.  Each vertex of the tree is a closed disk with rational radius of the form $p^{-n}$, and two vertices are connected by a branch if the disks have consecutive radii (see Section \ref{setting} for details).  Just as Hubbard trees are defined in $\mathbb{C}$, we define the \textit{critical orbit tree} for a polynomial $f\in\mathbb{Z}_p[z]$ as the convex hull of the critical orbit and the induced dynamical action on this sub-tree of $\mathbb{Z}_p$.  The \textit{critical orbit tree (mod $p$)} is the subset consisting of the residue classes (mod $p$) and the induced dynamical system on this restricted tree.  Theorem \ref{preper} and the previous work done on the periodic case imply the following about critical orbit trees in $\mathbb{Z}_p$:

\begin{theorem}\label{tree}
Let $p\geq 3$ and suppose 0 has finite orbit for $f_c(z)=z^2+c$, $c\in\mathbb{Z}_p$. 
\begin{enumerate}
\item If 0 is periodic of exact period $n$, the critical orbit tree coincides with the critical orbit tree (mod $p$).  It is a finite tree with a single vertex of degree $n$, and $f_c$ acts on the $n$ end points by a cyclic permutation.
\item If 0 is strictly pre-periodic with orbit type $(m,n)$, the critical orbit tree either coincides exactly with the critical orbit tree (mod $p$) or it differs by one instance of branching.  
\end{enumerate}

\end{theorem}

Motivated by results in complex dynamics, this article is the first step in a project to obtain a complete description of all possible critical orbit trees for parameters $c$ in $\mathbb{Z}_p$, $p\geq2$.

The complete list of finite critical orbit trees for $p=3$ is shown in Figure \ref{3trees}.  Note that the $(2,3)$ tree is the same as the $(2,1)$ tree (mod $3$), and then branches once below (mod $3^2)$.  We provide the complete list of finite critical orbit trees for $p=5$ and $p=7$ in Section \ref{small}.

\begin{figure}[h]
\centerline{ \includegraphics[width=3.5in]{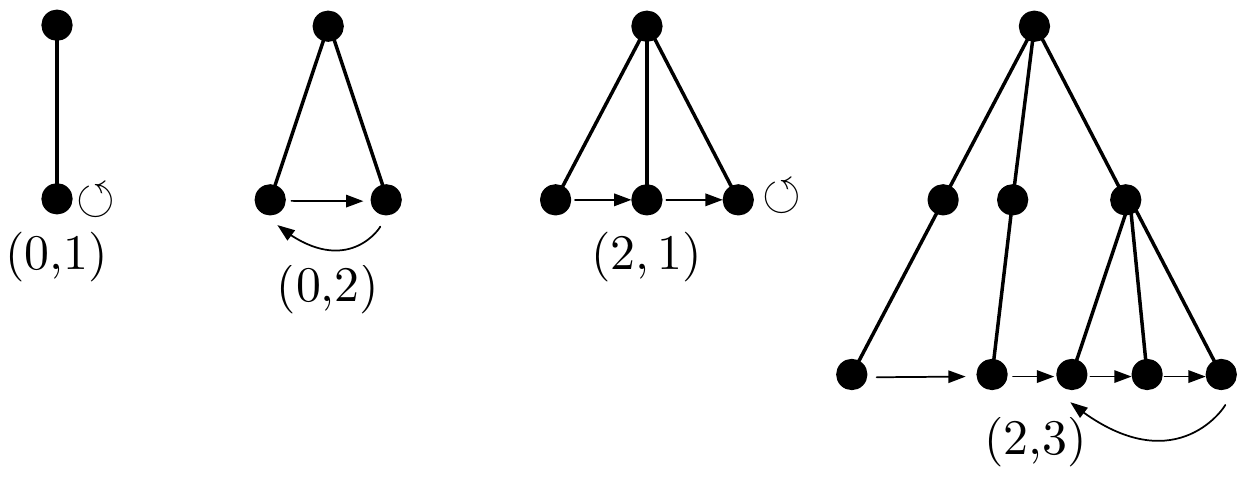}}
 \caption{The 4 distinct finite critical orbit trees in $\mathbb{Z}_3$, with corresponding critical orbit type labeled below.  Note that the trees are only drawn (mod $p^k$) when the full tree in $\mathbb{Z}_p$ agrees with the truncated tree.  In this case we say that the orbit is \textit{resolved} (mod $p^k$).\protect\label{3trees}}
\end{figure}

We are not only interested in the structure of finite critical orbit trees in $\mathbb{Z}_p$, however.  Using the technique of linearization outlined in \cite{Olof}, it is possible to determine the structure of infinite critical orbit trees corresponding to parameters near PCF points in $\mathbb{Z}_p$, based on the proximity of the PCF parameters.  In this article we look closely at the dynamics of $\mathbb{Z}_3$ parameters inside $\mathcal{M}_3$, near the PCF point $c=-2$.  In Section \ref{p3} we prove the following theorem:

\begin{theorem}\label{c-2}
Let $p=3$ and consider the map $f_c(z)=z^2+c$ for $c\in\mathbb{Z}_3$.  In $D(-2, 1/9)$, if a parameter $c$ is such that 
\[c\equiv -2 \mbox{ (mod } 3^k) \qquad \mbox{ but } \qquad c\not\equiv -2 \mbox{ (mod }3^{k+1}),\]
then 0 has orbit type $(2,3^i)$ (mod $3^{k+i}$) for $i\geq 1$, for the corresponding map $f_c(z) = z^2 + c$.
\end{theorem}

The structure of this article is as follows.  In Section \ref{setting} we describe the setting.  In Section \ref{tools} we include some technical results that will be used to prove Theorems \ref{preper} and \ref{tree}.  We prove Theorems \ref{preper} and \ref{tree} in Section \ref{proof} and include a discussion of existing work.  Section \ref{small} contains the complete list of finite critical orbit trees for $p=5$ and $p=7$, and in Section \ref{p3} we prove Theorem \ref{c-2}.  We conclude the paper with a brief discussion of the case of $p=2$ in Section \ref{p2}. 

\subsection*{Acknowledgements}  The author would like to thank her advisor, Laura DeMarco, for her tireless guidance and supervision, and Rob Benedetto, Ben Hutz, and Rafe Jones for their many helpful email exchanges which all contributed to this project.

\section{The Setting}\label{setting}

Throughout this article, we exploit the natural tree structure of $\mathbb{Z}_p$, which may be viewed as a subtree of the Berkovich projective line over $\mathbb{C}_p$.  Each point, or vertex, of the $\mathbb{Z}_p$ tree is a closed disk with rational radius belonging to the valuation group $|\mathbb{Z}_p^*| = \{p^{-n}\}$, where $|\cdot|$ is a non-Archimedean absolute value.  We denote these points by $D(a, r) = \{x: |x-a|\leq r\}$, or simply by a center, $a$, if the radius is clear.  Two points $D(a_1,r_1)$ and $D(a_2,r_2)$, with $r_1> r_2$, are connected by a line segment, or branch, if they are such that $a_2\equiv a_1$ (mod ${r_1}^{-1}$).  Every point $D(a,r)$ other than the root of the tree has $p+1$ edges branching off of it, one for each of the $p$ elements of the residue field $\mathbb{Z}_p/p\mathbb{Z}_p\simeq \mathbb{F}_p$ lying below it, and one connecting it to the point $D(\bar{a},pr)$ above it.

Because $\mathbb{Z}_p = \{ \alpha \in\mathbb{Q}_p : |\alpha| \leq 1\}$, we may consider the top of the tree, the root, as the closed disk of radius 1 centered at 0, known as the Gauss point in Berkovich space parlance.  If one were to follow each branch down to the bottom of the tree, one would find all of the elements of $\mathbb{Z}_p$, each at the end of exactly one branch of the tree.  See \cite{Silverman} to compare Figure \ref{Z3tree} to the Berkovich projective line.

\begin{figure}[h]
\centerline{ \includegraphics[width=6in]{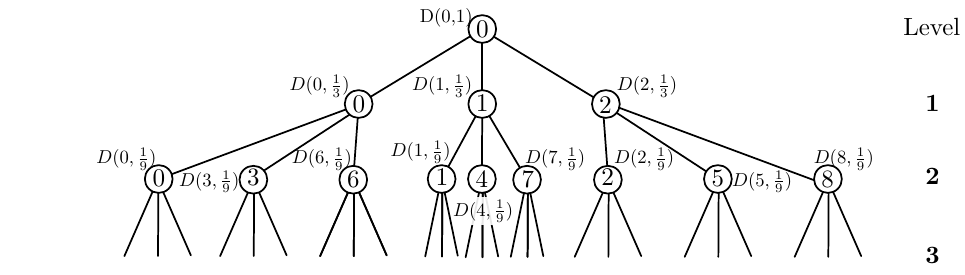}}
 \caption{The top of the $\mathbb{Z}_3$ tree.\protect\label{Z3tree}}
\end{figure}

We are interested not only in the action of our map $f_c(z) = z^2 + c$ on the ends of the tree ($\mathbb{Z}_p$), but also on the vertices.  By \cite{Juan} (Proposition 2.4), this action is well defined.  The induced map on the vertices, which is the reduction map of $f_c$ after a change of coordinates, will have positive degree.  We may then study the local dynamics of $f_c$ by considering the iterates of this induced map.  See Lemma \ref{affine} for more details.

The $\mathbb{Z}_p$ tree has an induced metric defined on the vertices $\zeta = D(a,r)$.  We define the combinatorial distance between 2 points connected by an edge, $\zeta_1= D(a_1, r_1) \supseteq \zeta_2= D(a_2,r_2)$ by $\rho(\zeta_1, \zeta_2) = log_p(r_1) - log_p(r_2)$.  Recall that in a non-Archimedean setting, any point in a disk may be considered as the center, so that $D(a_2, r_1) = D(a_1, r_1)$, for example.  This notion of distance may be extended to 2 points not connected by an edge by defining $\zeta_1 \vee \zeta_2$ to be the smallest disk that contains both $\zeta_1$ and $\zeta_2$.  Then $\rho(\zeta_1, \zeta_2) =\rho(\zeta_1, \zeta_1 \vee \zeta_2) + \rho(\zeta_2, \zeta_1 \vee \zeta_2)$.  This is  also known as the path metric on Berkovich hyperbolic space \cite{Baker}.

Lastly, when we refer to the ${k}^{th}$ \textit{level} of the $\mathbb{Z}_p$ tree, this includes all of the disks of radius $p^{-k}$, of which there are $p^k$.  The \textit{next level} would then be the ${k+1}^{st}$ level, the disks of radius $p^{-(k+1)}$ lying directly below those of radius $p^{-k}$.

\section{Preliminary Results}\label{tools}

This section contains the technical machinery necessary for the proofs of Theorems \ref{preper} and \ref{tree}, and for a discussion of the previous work done on the (mod $p$) periodic case.  Some results are adapted from other papers (see \cite{Benedetto}, \cite{Jones}, \cite{Morton}, \cite{Morton2}, \cite{Pezda}, and \cite{Juan}), and the language may have been modified slightly for our purposes.  In general, fix a prime $p\geq3$, and consider  $f_c(z)=z^2+c$, for $c\in\mathbb{Z}_p$.  Some of the following will hold for parameters in spaces larger than $\mathbb{Z}_p$, and other results may also apply to $p=2$.  Any exceptions are indicated.

This first result is included in \cite{Jones} Corollary 1.5 (see Lemma \ref{jones1.5} below).  We give an alternate proof here, to highlight the fact that the attracting cycle is in the same residue class as 0.

\begin{lemma}\label{attracting}
If 0 is periodic of exact period $n$ under iteration of $f_c$ (mod $p$), then there exists an attracting cycle of exact period $n$ for $f_c$ over $\mathbb{Z}_p$.
\end{lemma}

\begin{proof}
Fix $f_c(z)=z^2+c$ and suppose that $f_c^n(0)\equiv 0$ (mod $p$), so 0 is a root of $g(z) = f^n(z)-z$ (mod $p$).  The derivative is $g'(z) = (f^n)'(z) - 1$, so $g'(0) = (f^n)'(0) - 1 \equiv p - 1$ (mod $p$).  Since $p-1$ will always be nonzero (mod $p$), the conditions for Hensel's lemma are satisfied.  There exists a unique $w\in \mathbb{Z}_p$ such that $g(w) = 0$ and $w\equiv 0 $ (mod $p$).  So $f^n(w) - w = 0$, and thus $w$ is periodic of period $m$ such that $m|n$.  It remains to show that $m=n$.

However, for any $m<n$ with $m|n$, we must have $f^m(w)-w\neq 0$ since $n$ is the least positive integer such that $f^n(0)\equiv 0 $ (mod $p$).  If $f^m(w) - w = 0$ with $w\equiv 0 $ (mod $p$), then $f^m(0) \equiv 0 $ (mod $p$), a contradiction.  So $w$ is periodic of exact period $n$.

Further, $w\equiv 0$ (mod $p$) implies that $|w|_p<1$.  Since $w| (f^n)'(w)$, we have $|(f^n)'(w)|_p <1$.  So $w$ is an attracting periodic point, of exact period $n$.
\end{proof}

Calculating the critical orbit of $f_c$ (mod $p$) is equivalent to calculating the critical orbit of the reduced map $\tilde{f_c}$ for $c \in \mathbb{Z}_p$.  Thus, the following result (\cite{Jones} Corollary 1.5) implies that an attracting cycle exists only for those $f_c$ such that the critical point is periodic (mod $p$).

\begin{lemma}\emph{\textbf{(Jones \cite{Jones})}}\label{jones1.5}
Let $c \in D=\{\alpha\in\mathbb{C}_p: |\alpha|\leq 1\}$. Then $f_c$ has a unique attracting cycle if 0 is periodic under iteration of $\tilde{f_c}$ and no attracting cycles otherwise. Moreover, if 0 is periodic of primitive period $n$ under iteration of $\tilde{f_c}$, then the unique attracting cycle of $f_c$ has primitive period $n$.
\end{lemma}

Lemma \ref{jones1.5} together with the following theorem, adopted from Proposition 4.6 of \cite{Juan}, illustrates why the critical cycle length remains fixed at $n$ when calculated mod higher powers of $p$, even when the critical orbit is infinite: because 0, the only critical point of $f_c$, is in the basin of attraction of the attracting $n$-cycle.

\begin{theorem}\emph{\textbf{(Rivera-Letelier \cite{Juan})}}\label{rivera}
Suppose a map $\varphi\in\mathbb{C}_p(z)$ has good reduction. Then if $p$ does not divide deg$(\varphi)$ then the immediate basin of attraction of any attracting cycle contains a critical point of $\varphi$.

\end{theorem} 

This brings up the notion of good reduction for a map $\varphi: \mathbb{P}^1(\mathbb{C}_p) \rightarrow \mathbb{P}^1(\mathbb{C}_p)$ with coefficients in $\mathcal{O}_{\mathbb{C}_p}$.  The reduced map $\tilde{\varphi}$ is obtained by reducing each coefficient modulo the maximal ideal of $\mathcal{O}_{\mathbb{C}_p}$, $\mathfrak{p}=\{\alpha \in \mathbb{C}_p : |\alpha|<1\}$.  We then say that $\varphi$ has \textit{good reduction} if deg($\varphi$) = deg($\tilde{\varphi})$.  For more details, see (\cite{Silverman} \S 2.3).  Note that if $f_c$ is defined over $\mathbb{Z}_p$ then it has good reduction at all primes $p$.

Lastly, by (\cite{Benedetto} Theorem 1.5), we see that the critical point is either in the attracting cycle, or it has infinite orbit and is attracted to the cycle.  The critical point will not eventually land on the attracting cycle.

\begin{theorem}\emph{\textbf{(Benedetto, Ingram, Jones, and Levy \cite{Benedetto})}}\label{benedetto}
Let $K$ be an algebraically closed field which is complete with respect to a non-trivial non-Archimedean absolute value $|\cdot|$.  Let $p\geq 0$ be the residue characteristic of $K$, and suppose either that $p=0$ or that $\varphi \in K(z)$ has degree $d<p$.  If $\gamma$ is an $n$-periodic point of $\varphi$ with 
\[
0<|(\varphi^n)'(\gamma)|<1,
\]
then there is a critical point of $\varphi$ with infinite orbit that is attracted to the cycle containing $\gamma$.  
\end{theorem}

Thus ends the machinery necessary to prove Proposition \ref{per} in the next section, for periodic critical points (mod $p$).  For strictly pre-periodic critical points, we have the following:

\begin{lemma}\label{affine}
If 0 is strictly pre-periodic under iteration of  $f_c$ (mod $p$), then the iterates of $f_c$ behave locally as linear affine transformations, and $f_c$ is a local isometry.
\end{lemma}

\begin{proof}

Define $g$ to be the reduction of an iterate of $f_c$ after a change of coordinates, so that $g$ fixes some point $v$ in the orbit of $0$.  We choose $v$ to be the start of the periodic cycle, $f^m(0)$, when the critical orbit is calculated (mod $p^k$).  If $v$ has reduced cycle length $n$, then $g$ is the reduction of $f^n$.  Under the change of coordinates, 
\[g(z) = \frac{1}{p^k}(f^n(v+p^k\cdot z) - v)\qquad (\mbox{mod } p).\]

Because $f$ has good reduction, the local map $g$ will be of degree at most 2.  We may rule out the possibility of deg($g) =2$, however, since we have chosen a vertex away from the critical point.  By (\cite{Juan} Proposition 2.4) we know it has positive degree, so deg$(g)=1$.   Thus, $g(z) = az+b$ for some $a,b\in\mathbb{F}_p$.

Further, we can show that $f_c$ is a local isometry on vertices away from the critical point.  The distance between two points is preserved:  Choose $\alpha_1, \alpha_2\in\mathbb{Z}_p$ such that $|\alpha_i|=1$ for $i=0,1$.  Suppose $|\alpha_2 - \alpha_1| = p^{-r}$, so without loss of generality we may write $\alpha_2 = \alpha_1 + a\cdot p^r + \displaystyle \sum_{i=1}^{\infty} a_i p^{r+i}$, with $a \in\mathbb{F}_p^*$.  Then $|f(\alpha_2)-f(\alpha_1)|=|\alpha_2^2 + c - (\alpha_1^2+c)|=|\alpha_1^2 + 2\alpha_1a\cdot p^r + O(p^{r+1}) - \alpha_1^2| = |p^r|| 2\alpha_1a+O(p)|=p^{-r} $ since $|2a\alpha_1|=1$.

\end{proof}

The next proposition demonstrates that when 0 is strictly pre-periodic of orbit type $(m,n)$ (mod $p$), the tail length remains fixed at $m$ when the orbit is calculated modulo higher powers of $p$, even if the orbit is infinite over $\mathbb{Z}_p$.

\begin{proposition}\label{fixed tail}
Suppose that 0 has strictly pre-periodic orbit type $(m,n)$ (mod $p$) for $f_c$, with $c$ in a finite extension of $\mathbb{Q}_p$.  Then 0 will have orbit type $(m, n_k)$ (mod $p^k$) for all $k\geq 1$, where $n_k$ is the length of the cycle that 0 lands on when its orbit is calculated (mod $p^k$).  
\end{proposition}

\begin{proof}

First note that the tail length $m$ cannot shrink as the orbit type of 0 is calculated modulo higher powers of $p$, and so we need only to consider the possibility of the tail length growing.  We will show that the $m^{th}$ iterate of 0 is periodic at every level, by induction on the power of $p$.

Consider the disk of radius $1/p$ around $f^m(0)$.  Let $\zeta=D(f^m(0), 1/p)$, and observe that the $n^{th}$ iterate of $f$ fixes the entire disk $\zeta$:  Because $f^n(f^m(0)) = f^m(0)$, and $f$ preserves the radius of the disk (see the proof of Lemma \ref{affine} above), we have $f^n: \zeta \longrightarrow \zeta$.

In local coordinates, $f^n_{\zeta}$ is a linear map (see Lemma \ref{affine}), so, in particular, it is a bijection on its residue classes.  Thus, the $\zeta$-class of $f^m(0)$ is periodic (mod $p$), and so $f^m(0)$ is periodic (mod $p^2$). 

Now assume that $f^m(0)$ is periodic of period $n_k$ (mod $p^k$), to show that it is periodic (mod $p^{k+1}$):  Let $\zeta_k = D(f^m(0), 1/p^k)$, and again observe that the $n_k^{th}$ iterate of $f$ fixes $\zeta_k$.  By Lemma \ref{affine}, $f^{n_k}_{\zeta}$ will be a linear map in local coordinates, so it is a bijection on the residue classes below $\zeta_k$.  Thus, the $\zeta_k$-class of $f^m(0)$ is periodic (mod $p^k$), and so $f^m(0)$ is periodic (mod $p^{k+1}$).
\end{proof}

The following is an adaptation of Theorem 2 from \cite{Pezda}, which will be used in calculating the examples in Section \ref{small}.  It explains why we see only particular finite critical orbit types for given primes: it is not possible to construct a polynomial with critical orbit type $(m,n)$ for all positive integers $m,n$, when working in $\mathbb{Z}_p$.  There are certain restrictions on the cycle length, $n$, in addition to the restriction that $m<p$.

In order to state the theorem it is necessary to describe what is referred to as a $(*)$-cycle in \cite{Pezda}.  If $x_0, x_1, \ldots, x_{k-1}$ is a $k$-cycle in $\mathbb{Z}_p$, we say that it is a $(*)$\textit{-cycle} if $|x_i-x_j|<1$ for all $i\neq j$.  This equates to the cycle reducing to a single point (mod $p$).

\begin{theorem}\emph{\textbf{(Pezda \cite{Pezda})}}\label{pezda2}
A $(*)$-cycle of length $n$ exists in $\mathbb{Z}_p$ if and only if $n$ is a divisor of $p-1$ except for $p=2,3$, in which case $n$ can be any integer not exceeding $p$.
\end{theorem}

\begin{proof}

This is the culmination of preliminary results from \cite{Pezda}, which are listed below.
\begin{itemize}
\item (\cite{Pezda} Lemma 6):
If $m$ is the length of a $(*)$-cycle in $\mathbb{Z}_p$ and $p\nmid m$, then $m|(p-1)$.

\item (\cite{Pezda} page 15 Proposition):
If there is a $(*)$-cycle of length $p^{\alpha}$, then $\alpha \leq 1$.

\item (\cite{Pezda} Lemma 9):

\begin{enumerate}
\item If $x_0, x_1, \ldots, x_{p-1}$ is a $(*)$-cycle in $\mathbb{Z}_p$ and $ |x_1-x_0| = p^{-d}$ then $(p-2)d\leq 1$.
\item If $p>3$ then there are no $(*)$-cycles of length $p$ in $\mathbb{Z}_p$.  In $\mathbb{Z}_3$ there are no $(*)$-cycles of length 9 and in $\mathbb{Z}_2$ there are no $(*)$-cycles of length 4.
\end{enumerate}

\item (\cite{Pezda} Lemma 10):
There are no $(*)$-cycles of length 6 in $\mathbb{Z}_3$.

\end{itemize}
Then since any $(*)$-cycle of composite length $k=ab$ for a map $f$ corresponds to a $(*)$-cycle of length $a$ for $f^b$, or a $(*)$-cycle of length $b$ for $f^a$, we get the statement of the theorem.  
\end{proof}

One consequence of Pezda's work is the following corollary, which explains why there are indeed a very restricted number of possible finite critical orbit types that arise in $\mathbb{Z}_p$.

\begin{corollary}\label{mult}
Let $p\geq 3$ and consider the critical orbit for $f_c(z)=z^2+c$, $c\in\mathbb{Z}_p$.  Suppose 0 has orbit type $(m,n)$ (mod $p$), $m\geq 1$.  Then the cycle length can multiply by $r|(p-1)$ at most once ($r$ may equal $p$ if $p=3$), so that 0 has finite orbit type ($m,rn)$.  If the cycle length multiplies again, 0 has infinite orbit.
\end{corollary}

\begin{proof}
If 0 has finite orbit that is not resolved (mod $p$), it will be resolved when the orbit is calculated (mod $p^i$) for some $i\geq2$.  In this case each of the $n$ members of the (mod $p$) cycle has a $(*)$-cycle underneath it, when visualized in the $\mathbb{Z}_p$ tree.  Let the length of each $(*)$-cycle be $r$, with $r|p-1$ (or $r=3$ if $p=3$) by Theorem \ref{pezda2}.  The total length of the cycle will then be $rn$, and by (\cite{Pezda} Lemma 3) that is the maximum finite length allowed.  The length of any cycle in $\mathbb{Z}_p$ is $k=ab$, where $a$ is the length of a $(*)$-cycle and $b\leq p$.  If the length of the critical orbit cycle multiplies again when calculated (mod $p^j$), $j>i$, then the critical orbit must be infinite in $\mathbb{Z}_p$.
\end{proof}

Lastly, it is important to mention the following theorem of Morton and Silverman, which first appeared in \cite{Morton} (Theorem 1.1), although the proof is largely based on their work in \cite{Morton2}.  It provides some insight into understanding Corollary \ref{mult}, and shows that we not only have this phenomenon in $\mathbb{Z}_p$, but also in finite extensions of $\mathbb{Q}_p$.
  
\begin{theorem}\emph{\textbf{(Morton and Silverman \cite{Morton})}}\label{MS}
Let $K/\mathbb{Q}_p$ be a $p$-adic field with maximal ideal $\mathfrak{p}$, let $\varphi: \mathbb{P}^1 \rightarrow \mathbb{P}^1$ be a rational map of degree at least two defined over $K$ and with good reduction at $\mathfrak{p}$, and let $P\in\mathbb{P}^1(K)$ be a periodic point for $\varphi$ of minimal period $n$.  Define integers $m$ and $r$ by
\[
\begin{array}{rl}
m = & \tilde{\varphi} \mbox{-period of } \tilde{P} \mbox{ in } \mathbb{P}^1(\mathbb{F}_p),\\
r = & \mbox{multiplicative period of } (\tilde{\varphi}^m)'(\tilde{P}) \mbox{ in } \mathbb{F}_p^*.
\end{array}
\]
(If $(\tilde{\varphi}^m)'(\tilde{P}) = 0$, we set $r=\infty$.)  Then either 
\[
\begin{array}{l}
n =m, \mbox{ } n=mr, \mbox{ or } n=mrp^e \quad \mbox{ for some } e\geq 1.
\end{array}
\]
\end{theorem}

\section{The Structure of the Critical Orbit Calculated (mod $p^i$)}\label{proof}

We begin this section with a discussion of the known case of when 0 is periodic (mod $p$) for $f_c$.

\begin{proposition}\label{per}
Let $p\geq 3$ and consider the critical orbit for $f_c(z)=z^2+c$, $c\in\mathbb{Z}_p$.  If 0 has orbit type $(0,n)$ (mod $p$), then either 0 is periodic of exact period $n$ or 0 has infinite orbit, with orbit type $(m_i, n)$ (mod $p^i$) for all $i\geq 1$.
\end{proposition}

\begin{proof}

Proposition \ref{per} follows directly from the results given in Section \ref{tools}.  Since 0 is periodic (mod $p$), there is an attracting cycle of period $n$ for $f_c$ by Lemma \ref{attracting}.  In fact, from the proof of this lemma, it is evident that the attracting point is in the same residue class as 0.  Similarly, the other $n-1$ elements of the attracting cycle are in the same classes as the elements of the reduced critical orbit (mod $p$), for $n>1$.  

This is the only attracting cycle for $f_c$, of any cycle length, by Corollary \ref{jones1.5}.  By Theorem \ref{rivera}, 0 is in the basin of attraction of that attracting cycle, since it is the only critical point for $f_c$. And lastly, if 0 is not actually periodic over $\mathbb{Z}_p$, it will remain strictly attracted to the attracting cycle and never eventually join it, by Theorem \ref{benedetto}.  So if 0 does not have exact period $n$, it will have infinite orbit, but at every level of the $\mathbb{Z}_p$ tree, 0 will be pre-periodic with orbit type $(m_i, n)$.  The critical point will always land on exactly one representative of the residue class for each of the $n$ elements in the attracting cycle, and remain in that cycle once it gets there, but it will take longer and longer to reach the cycle as the tail length grows when the orbit is calculated modulo higher powers of $p$.
\end{proof}

Now we consider the case of when 0 is strictly pre-periodic (mod $p$) for $f_c$.

\subsection*{Theorem \ref{preper}.} \textit{Let} $p\geq 3$ \textit{and consider the critical orbit for} $f_c(z)=z^2+c$, $c\in\mathbb{Z}_p$\textit{.  If 0 is strictly pre-periodic with orbit type} $(m,n)$ \textit{(mod }$p$), $m>0$\textit{, then either 0 has orbit type} $(m,n)$ \textit{over} $\mathbb{Z}_p$ \textit{or there exists some }$k\geq 1$ \textit{and} $r|(p-1)$ \textit{in} $\mathbb{Z}$ \textit{such that} 
\begin{enumerate}
\item \textit{0 has orbit type} $(m,n)$ \textit{(mod }$p^i$\textit{) for all} $i\leq k$\textit{, and} 
\item \textit{0 has orbit type} $(m, rn)$ \textit{(mod} $p^j$\textit{) for all} $j> k$ \textit{, with} $r=p$ if $p=3$.
\end{enumerate}
\textit{Otherwise, 0 has infinite orbit, with orbit type} $(m, n_i)$ \textit{(mod} $p^i$\textit{) for all} $i\geq 1$\textit{, where} $n_i$ \textit{is the length of the cycle on which 0 lands when its orbit is calculated (mod }$p^i$\textit{).}

\begin{proof}

Theorem \ref{preper} also follows directly from the results in Section \ref{tools}.  Since 0 is pre-periodic (mod $p$), it will have orbit type $(m, n_i)$ (mod $p^i$) for all $i\geq 1$ by Proposition \ref{fixed tail}.  The tail length will be fixed at $m$ when the critical orbit is calculated modulo higher powers of $p$, even if the orbit is infinite.

If the critical orbit is finite, by Corollary \ref{mult} the (mod $p$) cycle length may multiply by some number at most once.  That number $r$ must divide $(p-1)$ because the resulting cycle lying below each element of the original (mod $p$) cycle will be a $(*)$-cycle, unless $p=3$, in which case $r$ may be 3.  If the cycle length does multiply by $r$, there will be some maximal power of $p$, say $p^k$, such that the critical orbit type is $(m,n)$ (mod $p^k$) and $(m, rn)$ (mod $p^{k+1}$).  The cycle length will not increase again, so the critical orbit type will be $(m, rn)$ (mod $p^{k+i})$ for all $i\geq 1$, and in $\mathbb{Z}_p$.

These are the only possibilities for strictly pre-periodic critical portraits.  When calculated in $\mathbb{Z}_p$, the critical orbit type will be ($m,1)$, ($m, n)$ with $n<p$, or $(m, rn)$.  If $p=3$, we may have $r=p=3$, by Theorem \ref{pezda2}; otherwise, $r|(p-1)$.
\end{proof}

\begin{remark}
By Theorem \ref{MS}, we know that $r$ is the multiplicative order of the multiplier of $\widetilde{f_c^m(0)}$ in ${\mathbb{F}_p}^*$.
\end{remark}

We are now ready to explore the implications of Proposition \ref{per} and Theorem \ref{preper} on the tree structures of finite critical orbits in $\mathbb{Z}_p$.

\subsection*{Theorem \ref{tree}.}  \textit{Let }$p\geq 3$ \textit{and suppose 0 has finite orbit for }$f_c(z)=z^2+c$, $c\in\mathbb{Z}_p$. 
\begin{enumerate}
\item \textit{If 0 is periodic of exact period }$n$\textit{, the critical orbit tree coincides with the critical orbit tree (mod }$p$\textit{).  It is a finite tree with a single vertex of degree }$n$\textit{, and }$f_c$\textit{ acts on the} $n$ \textit{end points by a cyclic permutation.}
\item \textit{If 0 is strictly pre-periodic with orbit type }$(m,n)$\textit{, the critical orbit tree either coincides exactly with the critical orbit tree (mod }$p$\textit{) or it differs by one instance of branching.}  
\end{enumerate}

\begin{proof}

1. By Proposition \ref{per}, if the orbit of 0 is periodic of exact period $n$, it will be periodic of exact period $n$ (mod $p$).  In particular, each of the $n$ elements of the cycle belongs to a distinct residue class (mod $p$), and so the structure of the critical orbit tree (mod $p$) will be exactly the same structure as the tree in $\mathbb{Z}_p$.  Further, the critical orbit tree will be a subtree of the top of the $\mathbb{Z}_p$ tree, which consists of a vertex (the Gauss point) of degree $p$.  (See Figure \ref{Z3tree}, through the first level.)  The critical orbit sub-tree will simply be the Gauss point with degree $n$ ($\leq p$).  And since the end points of the sub-tree are all of exact period $n$, the iterates of $f$ will act as cyclic permutations, with $f^n(z)$ as the identity.

2.  By Theorem \ref{preper}, if 0 is strictly pre-periodic with orbit type $(m,n)$, it will be strictly pre-periodic with orbit type $(m, n_1)$ (mod $p$), and either $n_1 = n$ or $rn_1=n$.  If $n_1=n$, then the critical orbit tree (mod $p$) is exactly the critical orbit tree in $\mathbb{Z}_p$, where each of the $m$ elements in the tail corresponds to a distinct residue class (mod $p$), and each of the $n$ elements in the cycle corresponds to a different distinct residue class (mod $p$).  If $rn_1=n$, then there are $r$ members of the cycle in each of the $n_1$ residue classes (mod $p$), so that each of the $n_1$ vertices (mod $p$) branches once, into $r$ distinct disks.  All $rn_1=n$ disks are in the same level in the tree.  In this case, the critical orbit tree (mod $p$) coincides with the tree in $\mathbb{Z}_p$ except for one instance of branching at the ${k+1}^{st}$ level, as in Theorem \ref{preper}.  The finite tree will have $m+n_1$ original branches from the top of the tree, so that the Gauss point has degree $m+n_1$, and if $n_1\neq n$ then each of the $n_1$ vertices in the periodic cycle will split into $r$ branches.  For example, compare the $(2,1)$ and $(2,3)$ trees in Figure \ref{3trees}.
\end{proof}

For a more explicit illustration of Theorem \ref{tree}, please see the examples of finite critical orbit trees in the next section.

\section{The $\mathbb{Z}_p$ PCF Points and Critical Orbit Trees for $p=3,5,7$}\label{small}

Given Theorem \ref{preper} and Proposition \ref{per}, it is relatively straightforward to calculate all possible PCF orbit types for a given prime, $p$.  We begin by calculating the $p$ distinct critical orbits for $f_c$ (mod $p$), with parameters in the $p$ residue classes (mod $p$).  Only the strictly pre-periodic (mod $p$) critical orbits need further work; the periodic critical orbits (mod $p$) indicate the existence of a hyperbolic component in the parameter space, which contains only the single isolated PCF parameter in the open disk of radius 1 in $\mathbb{Z}_p$.  (See the discussion in \S 1 of  \cite{Jones} regarding the size of hyperbolic components.)  The pre-periodic orbits are restricted in cycle length (see Theorem \ref{pezda2} and Corollary \ref{mult}), so once the cycle length has increased beyond the maximum allowed in $\mathbb{Z}_p$, we can be sure that the critical orbit is infinite for that class of parameters.

\begin{example}\label{p3ex}

Let $p=3$.  There are a total of 4 PCF parameters in $\mathbb{Z}_3$.
\begin{enumerate}
\item There are 2 hyperbolic components in $\mathbb{Z}_3$.  
\begin{enumerate}
\item For all $\mathbb{Z}_3$ parameters in the disk $D(0,1/3)$, the critical point is attracted to an attracting fixed point.  When $c=0$, 0 is that attracting fixed point, and for all other parameters 0 has infinite orbit. 
\[\mbox{The critical orbit for } f(z) = z^2 \mbox{ is } 0 \rightarrow 0 \mbox{ (mod } 3)\]
\item For all $\mathbb{Z}_3$ parameters in the disk $D(2,1/3)$, the critical point is attracted to an attracting 2-cycle. When $c=-1$, 0 is in that attracting cycle, and for all other parameters 0 has infinite orbit.
\[\mbox{The critical orbit for } f(z) = z^2 + 2  \mbox{ is } 0 \rightarrow 2 \rightarrow 0 \mbox{ (mod } 3)\]

\end{enumerate}

\item There exist 2 PCF points outside of the hyperbolic components, in the disk $D(1,1/3)$.  See Figure \ref{parameter3}.
\begin{enumerate}
\item When $c=-2$, 0 is pre-fixed with orbit type $(2,1)$.  All other parameters in $D(7,1/9)$ have infinite critical orbit, with orbit type $(2, 3^i)$ (mod $3^k$) for $k\geq 1$.  
\[\mbox{The critical orbit for } f(z) = z^2 - 2  \mbox{ is } 0 \rightarrow 1 \rightarrow 2 \rightarrow 2 \mbox{ (mod } 3)\]

\item There exists another PCF parameter $c \in D(1,1/9)$ with critical orbit type $(2,3)$.  All other parameters in the disk have infinite critical orbit, with orbit type $(2, 3^i)$ (mod $3^k$) for $k\geq 2$.  
\[\mbox{The critical orbit for } f(z) = z^2 + c  \mbox{ is } 0 \rightarrow 1 \rightarrow 2 \rightarrow 2 \mbox{ (mod } 3)\]
\[\mbox{The critical orbit resolves (mod } 3^2), \mbox{ with orbit type } (2,3):\]
\[ 0 \rightarrow 1 \rightarrow 2 \rightarrow 5 \rightarrow 8 \rightarrow 2 \]
\end{enumerate}
\end{enumerate}

Note that there are no PCF parameters in $D(4,1/9)$.  A more thorough discussion of the behavior around the PCF point $c=-2$ in $D(1,1/3)$ follows in the next section.

\end{example}

\begin{figure}[h]
\centerline{ \includegraphics[width=5in]{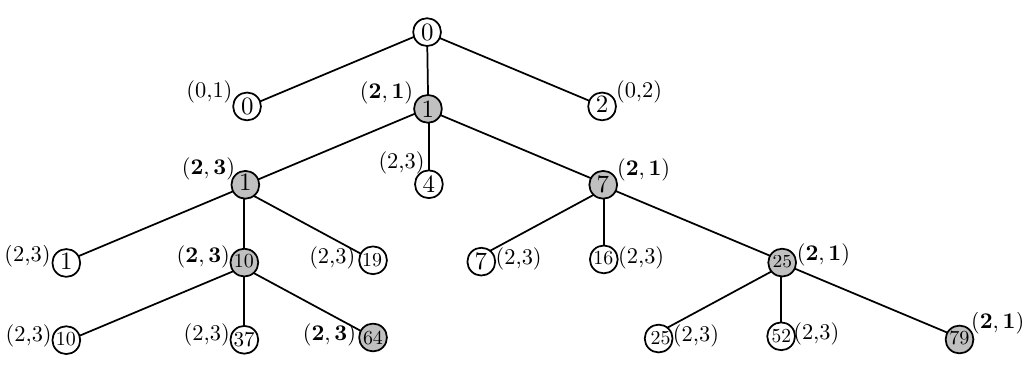}}
 \caption{$\mathbb{Z}_3$ as a parameter space, with the critical orbit type corresponding to each parameter, calculated at that level.  The shaded branch on the left follows the PCF parameter with critical orbit type $(2,3)$ and the shaded branch on the right follows the PCF parameter $c=-2$.  All other branches in $D(1,1/3)$ correspond to parameters with infinite critical orbit, with critical orbit type $(2, 3^i)$ at every level.\protect\label{parameter3}}  
\end{figure}

After the above calculations are completed for each PCF parameter, we translate the data to the corresponding critical orbit tree with the help of Theorem \ref{tree}.  If the critical point is of exact period $n$, the critical orbit tree will have a single vertex of degree $n$.  If the critical point is strictly periodic with orbit type $(m,n)$ (mod $p$), we start with the (mod $p$) tree, which will have one vertex of degree $m+n$.  This may be the whole tree if the (mod $p$) orbit type is the same as the critical orbit type in $\mathbb{Z}_p$, such as in the $(2,1)$ example above.  If not, the tree will have one instance of branching, at the level above which the $\mathbb{Z}_p$ critical orbit type is resolved.  In the $(2,3)$ example the critical orbit is $(2,3)$ (mod $3^2$), so the (mod 3) tree will branch at the first level.  The single vertex in the (mod 3) cycle will have 3 branches below it.  

See Figure \ref{3trees} for a complete list of the finite critical orbit trees in $\mathbb{Z}_3$.

\begin{example}
Let $p=5$.  There are a total of 7 PCF parameters in $\mathbb{Z}_5$.
\begin{enumerate}
\item There are 3 hyperbolic components in $\mathbb{Z}_5$.  
\begin{enumerate}
\item For all $\mathbb{Z}_5$ parameters in the disk $D(0,1/5)$, the critical point is attracted to an attracting fixed point.  When $c=0$, 0 is that attracting fixed point, and for all other parameters 0 has infinite orbit. 
\[\mbox{The critical orbit for } f(z) = z^2 \mbox{ is } 0 \rightarrow 0 \mbox{ (mod } 5)\]

\item For all $\mathbb{Z}_5$ parameters in the disk $D(4,1/5)$, the critical point is attracted to an attracting 2-cycle. When $c=-1$, 0 is in that attracting cycle, and for all other parameters 0 has infinite orbit.
\[\mbox{The critical orbit for } f(z) = z^2+4 \mbox{ is } 0 \rightarrow 4 \rightarrow 0  \mbox{ (mod } 5)\]

\item  There exists a parameter $c\in D(1,1/5)$ such that 0 is in an attracting 3-cycle.  For all other parameters in the disk, the critical point is attracted to an attracting 3-cycle.
\[\mbox{The critical orbit for } f(z) = z^2+1 \mbox{ is } 0 \rightarrow 1 \rightarrow 2 \rightarrow 0  \mbox{ (mod } 5)\]

\end{enumerate}

\item There exist 4 PCF points outside of the hyperbolic components, in the disks $D(2,1/5)$ and $D(3,1/5)$.
\begin{enumerate}
\item There exist 2 PCF parameters in $D(2,1/5)$, one $c$ with critical orbit type $(2,2)$ and one $c'$ with critical orbit type $(2,8)$.  All other parameters in the disk have infinite critical orbit, with orbit type $(2, 8\cdot5^i)$ (mod $5^k$) for $k\geq 2$.  
\begin{equation*}
\begin{gathered}
\mbox{The critical orbit for } f(z) = z^2+ c \mbox{ and } f(z) = z^2 + c' \mbox{ is } 0 \rightarrow 2 \rightarrow 3 \rightarrow 1 \rightarrow 3  \mbox{ (mod } 5)\\
\mbox{The critical orbit for }  f(z) = z^2 + c'  \mbox{ resolves (mod } 5^3), \mbox{ with orbit type } (2,8):\\
 0 \rightarrow 107 \rightarrow 56 \rightarrow 118 \rightarrow 31 \rightarrow 68 \rightarrow 106 \rightarrow 93 \rightarrow 6 \rightarrow 18 \rightarrow 56 
\end{gathered}
\end{equation*}
 
\item There exist 2 PCF parameters in $D(3,1/5)$, $c=-2$ with critical orbit type $(2,1)$, and another parameter $c'$  with critical orbit type $(2,2)$.  All other parameters in the disk have infinite critical orbit, with orbit type $(2, 2\cdot 5^i)$ (mod $5^k$) for $k\geq 2$.  
\begin{equation*}
\begin{gathered}
\mbox{The critical orbit for } f(z) = z^2-2 \mbox{ and } f(z) = z^2 + c' \mbox{ is } 0 \rightarrow 3 \rightarrow 2 \rightarrow 2  \mbox{ (mod } 5) \\
\mbox{The critical orbit for }  f(z) = z^2 + c'  \mbox{ resolves (mod } 5^2), \mbox{ with orbit type } (2,2):\\
 0 \rightarrow 18 \rightarrow 17 \rightarrow 7 \rightarrow 17 
\end{gathered}
\end{equation*}
\end{enumerate}
\end{enumerate}

See Figure \ref{5trees} for a complete list of the finite critical orbit trees in $\mathbb{Z}_5$.

\end{example}

\begin{figure}[h]
\centerline{ \includegraphics[width=3.5in]{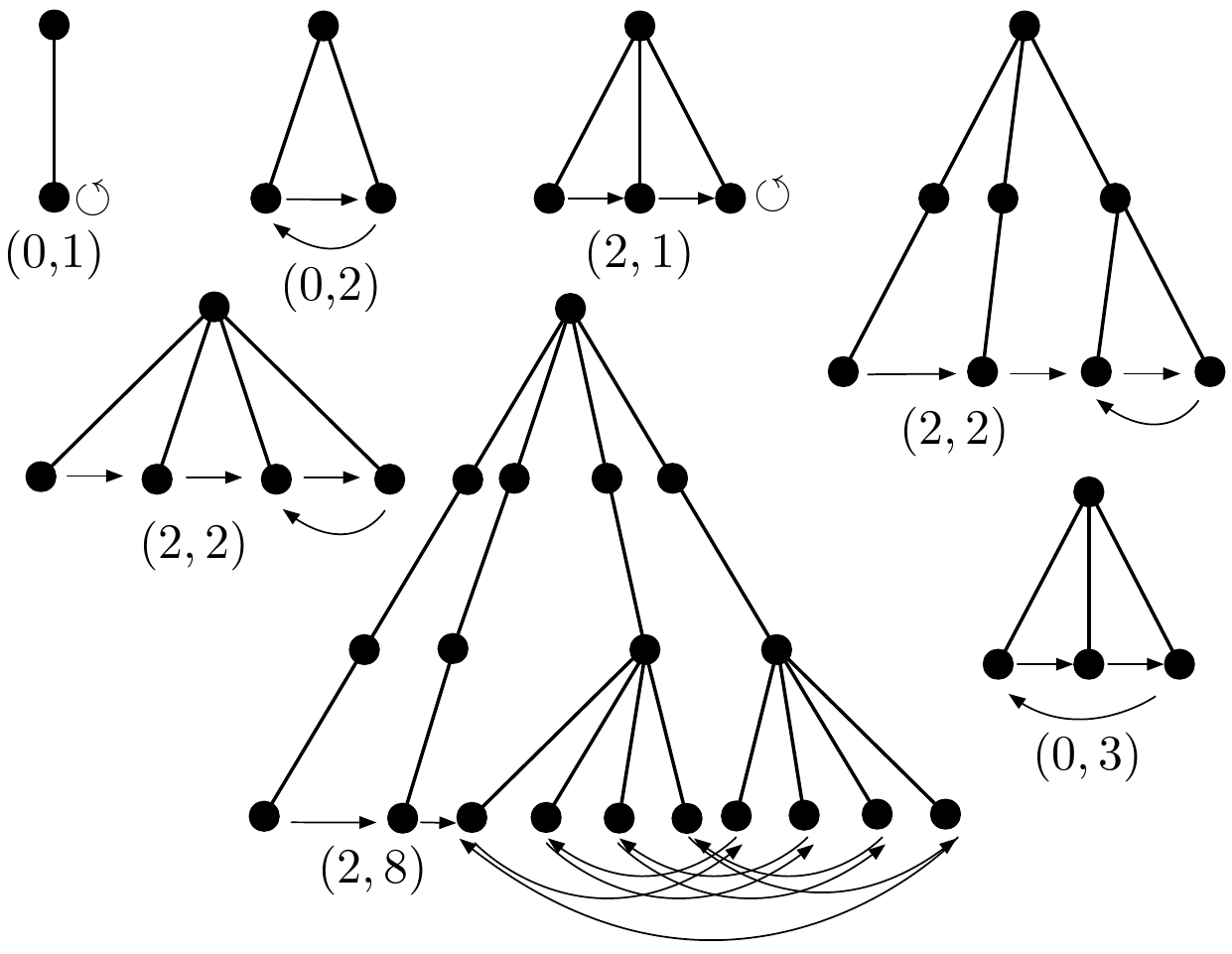}}
 \caption{The 7 distinct finite critical orbit trees in $\mathbb{Z}_5$, with corresponding critical orbit type labeled below.  Note that since the $(2,8)$ PCF orbit does not resolve until it is calculated (mod $5^3$), the branching occurs at the second level of the tree.\protect\label{5trees}}
\end{figure}

\begin{example}
Let $p=7$.  There are a total of 10 PCF parameters in $\mathbb{Z}_7$.
\begin{enumerate}
\item There are 3 hyperbolic components in $\mathbb{Z}_7$.  
\begin{enumerate}
\item For all $\mathbb{Z}_7$ parameters in the disk $D(0,1/7)$, the critical point is attracted to an attracting fixed point.  When $c=0$, 0 is that attracting fixed point, and for all other parameters 0 has infinite orbit. 
\[\mbox{The critical orbit for } f(z) = z^2 \mbox{ is } 0 \rightarrow 0  \mbox{ (mod } 7)\]

\item For all $\mathbb{Z}_7$ parameters in the disk $D(6,1/7)$, the critical point is attracted to an attracting 2-cycle. When $c=-1$, 0 is in that attracting cycle, and for all other parameters 0 has infinite orbit.
\[\mbox{The critical orbit for } f(z) = z^2 + 6 \mbox{ is } 0 \rightarrow 6 \rightarrow 0  \mbox{ (mod } 7)\]

\item  There exists a parameter $c\in D(3,1/7)$ such that 0 is in an attracting 3-cycle.  For all other parameters in the disk, the critical point is attracted to an attracting 3-cycle, and has infinite orbit.
\[\mbox{The critical orbit for } f(z) = z^2+ 3 \mbox{ is } 0 \rightarrow 3 \rightarrow 5 \rightarrow 0  \mbox{ (mod } 7)\]

\end{enumerate}

\item There exist 7 PCF points outside of the hyperbolic components, in the disks $D(1,1/7)$, $D(2,1/7)$, $D(4,1/7)$ and $D(5,1/7)$.
\begin{enumerate}
\item There exist 2 PCF parameters in $D(1,1/7)$, one $c$ with critical orbit type $(3,1)$ and one $c'$ with critical orbit type $(3,6)$.  All other parameters in the disk have infinite critical orbit, with orbit type $(3, 6\cdot7^i)$ (mod $7^k$) for $k\geq 2$.  
\begin{equation*}
\begin{gathered}
\mbox{The critical orbit for } f(z) = z^2+c  \mbox{ and } f(z) = z^2 + c' \mbox{ is } 0 \rightarrow 1 \rightarrow 2 \rightarrow 5  \rightarrow 5 \mbox{ (mod } 7) \\
\mbox{The critical orbit for }  f(z) = z^2 + c'  \mbox{ resolves (mod } 7^2), \mbox{ with orbit type } (3,6):\\
 0 \rightarrow 8 \rightarrow 23 \rightarrow 47 \rightarrow 12 \rightarrow 5 \rightarrow 33 \rightarrow 19 \rightarrow 26 \rightarrow 47
\end{gathered}
\end{equation*}

\item There exists 1 PCF parameter $c$ in $D(2,1/7)$, with critical orbit type $(4,1)$.  All other parameters in the disk have infinite critical orbit, with orbit type $(4, 7^i)$ (mod $7^k$) for $k\geq 1$.  
\[\mbox{The critical orbit for } f(z) = z^2+c  \mbox{ is } 0 \rightarrow 2 \rightarrow 6 \rightarrow 3  \rightarrow 4 \rightarrow 4 \mbox{ (mod } 7) \]

\item There exist 2 PCF parameters in $D(4,1/7)$, one $c$ with critical orbit type $(3,2)$ and one $c'$ with critical orbit type $(3,4)$.  All other parameters in the disk have infinite critical orbit, with orbit type $(3, 4\cdot7^i)$ (mod $7^k$) for $k\geq 2$.  
\begin{equation*}
\begin{gathered}
\mbox{The critical orbit for } f(z) = z^2+c  \mbox{ and } f(z) = z^2 + c' \mbox{ is } 0 \rightarrow 4 \rightarrow 5 \rightarrow 1  \rightarrow 5 \mbox{ (mod } 7) \\
\mbox{The critical orbit for }  f(z) = z^2 + c'  \mbox{ resolves (mod } 7^2), \mbox{ with orbit type } (3,4):\\
 0 \rightarrow 11 \rightarrow 34 \rightarrow 40 \rightarrow 43 \rightarrow 47 \rightarrow 15 \rightarrow 40
\end{gathered}
\end{equation*}

\item There exist 2 PCF parameters in $D(5,1/7)$, $c=-2$ with critical orbit type $(2,1)$, and another parameter $c'$ with critical orbit type $(2,3)$.  All other parameters in the disk have infinite critical orbit, with orbit type $(2, 3\cdot 7^i)$ (mod $7^k$) for $k\geq 2$.  
\begin{equation*}
\begin{gathered}
\mbox{The critical orbit for } f(z) = z^2-2  \mbox{ and } f(z) = z^2 + c' \mbox{ is } 0 \rightarrow 5 \rightarrow 2 \rightarrow 2  \mbox{ (mod } 7) \\
\mbox{The critical orbit for }  f(z) = z^2 + c'  \mbox{ resolves (mod } 7^2), \mbox{ with orbit type } (2,3):\\
 0 \rightarrow 26 \rightarrow 16 \rightarrow 37 \rightarrow 23 \rightarrow 16 
\end{gathered}
\end{equation*}

\end{enumerate}
\end{enumerate}

See Figure \ref{7trees} for a complete list of the finite critical orbit trees in $\mathbb{Z}_7$.

\end{example}

\begin{figure}[h]
\centerline{ \includegraphics[width=4in]{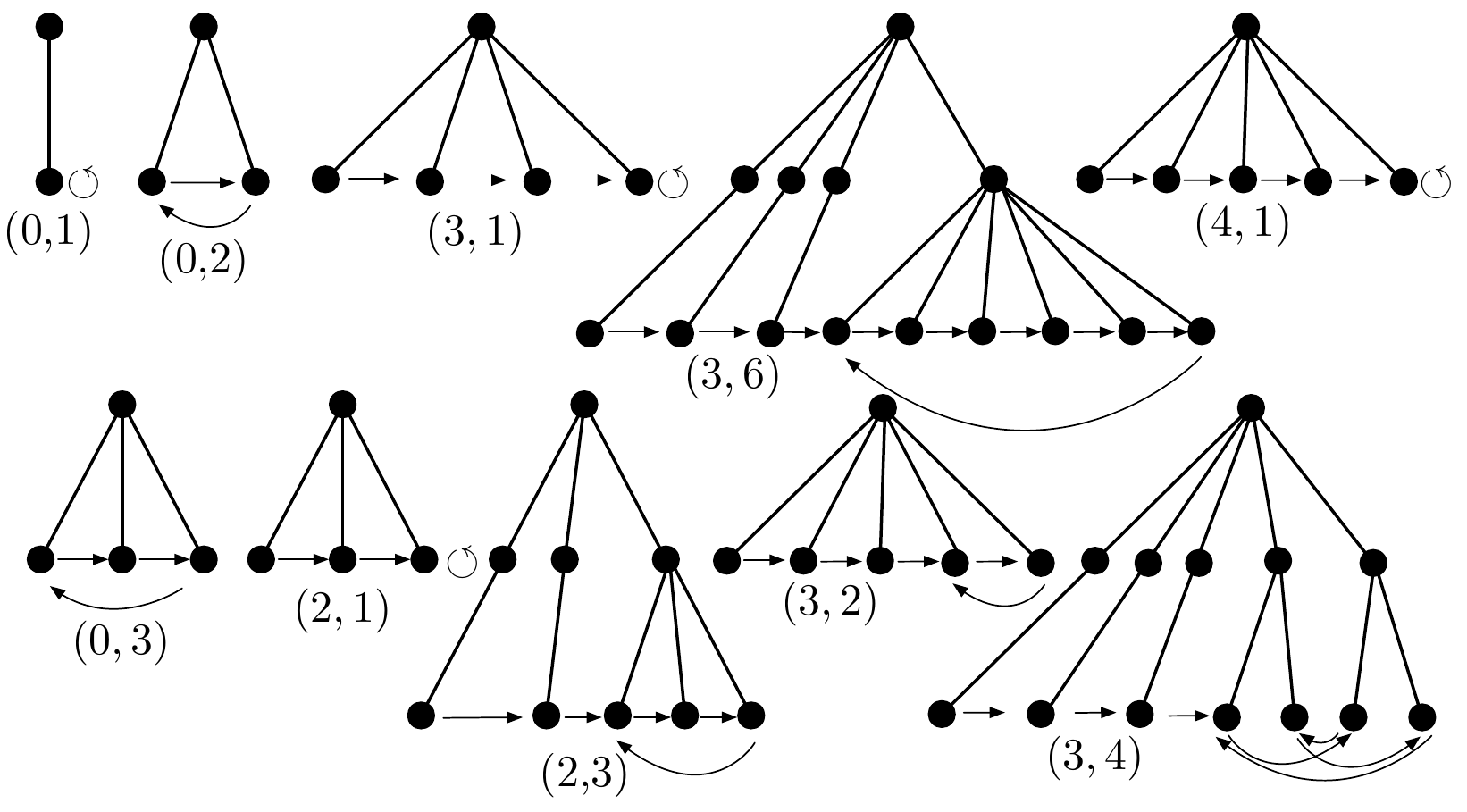}}
 \caption{The 10 distinct finite critical orbit trees in $\mathbb{Z}_7$, with corresponding critical orbit type labeled below. \protect\label{7trees}}
\end{figure}

We may replicate these calculations for any prime $p\geq 3$ in order to deduce exactly which finite critical orbit types arise for parameters in $\mathbb{Z}_p$.  There is a finite number of finite critical orbit types that are permissible for a given prime $p$, and in fact there is a finite number of PCF parameters in $\mathbb{Z}_p$. 

\begin{proposition}\label{number}
Fix $p\geq 3$ and let $Q_p$ be the number of quadratic PCF parameters in $\mathbb{Z}_p$.  Then  
\[
\displaystyle Q_p <  \sum_{\substack{m+n\leq p,\\ r|(p-1)}} 2^{m+rn},
\]
where the sum is taken over all triples $(m,r,n)$ such that $m+n\leq p$ and $r|(p-1)$, or $r=3$ if $p=3$.
\end{proposition}

\begin{proof}

For any parameter $c \in \mathbb{Z}_p$, $p\geq 3$, if the critical point has finite orbit for the corresponding map $f_c$, then the total length of that orbit is of the form $m+rn$, where $(m,n)$ is the (mod $p$) critical orbit type.  Since there are only $p$ residue classes (mod $p$), and each of the $m+n$ members of the (mod $p$) orbit must be distinct (mod $p$), $m+n\leq p$.  By Corollary \ref{mult}, the cycle length $n$ may multiply at most once by a number $r$ that divides $p-1$, or $r=3$ if $p=3$, so the cycle length in $\mathbb{Z}_p$ is $rn$, where $r$ could be 1.

Any parameter $c$ that corresponds to finite critical orbit type $(m, rn)$ will satisfy the equation $f^{m+rn}(0)= f^m(0)$.  Note that $f^{m+rn}(0)- f^m(0)$ is a degree $2^{m+rn}$ polynomial in $c$, and so there are a maximum of $2^{m+rn}$ possible distinct PCF parameters in $\mathbb{Z}_p$ that correspond to each given finite critical orbit type $(m, rn)$.  If we take the sum of $2^{m+rn}$ over all possible triples $(m,r,n)$ given the above restrictions on $m, r$, and $n$, then we have a very crude upper bound on the total possible number of PCF parameters that live in $\mathbb{Z}_p$, for given prime $p$.
\end{proof}

\begin{remark}
The number of non-zero hyperbolic components in $\mathbb{Z}_p$ is bounded above by $(p-1)/2$, as the PCF parameter must be a quadratic residue mod $p$.  Recall from Theorem \ref{per} that a hyperbolic component occurs only when the critical orbit is periodic (mod $p$).  This corresponds to the existence of a solution to  $f_c^n(0) = 0$ in $\mathbb{Z}_p$.  The point $z_0=f_c^{n-1}(0)$ satisfies $f_c(z_0) = {z_0}^2+c = 0$, so $z^2+c = 0$ must have a solution (mod $p$).  Thus $-c$ is a perfect square (mod $p$), so $-c$ is a quadratic residue (mod $p$), and there are exactly $(p-1)/2$ such values.  (See the discussion in \cite{Thiran}.)
\end{remark}

So we see that this bound on $Q_p$ may be refined quite a bit, if we even consider the periodic and strictly pre-periodic PCF parameters separately.  Proposition \ref{number} suffices to show finiteness, however, and so we will leave it at that.

Also note that in practice, there are not always $(p-1)/2$ hyperbolic components, as in the case of $p=7$.  In fact, as Rafe Jones conjectured in email correspondence, it seems that as $p$ increases to infinity, the proportion of hyperbolic components tends to 0. 

\section{The Dynamics of Parameters in $\mathbb{Z}_3$ near $c=-2$}\label{p3}

From Example \ref{p3ex}, we know that there are a total of 4 PCF parameters in $\mathbb{Z}_3$.  In this section we more closely examine the dynamics of maps corresponding to parameters in the disk $D(-2,1/9)$ around the PCF point $c=-2$.  We give the exact critical orbit type for each parameter, at every level of the $\mathbb{Z}_3$ tree, given how close the parameter is to $c=-2$.

\subsection*{Theorem \ref{c-2}.} \textit{Let }$p=3$\textit{ and consider the map }$f_c(z)=z^2+c$\textit{ for }$c\in\mathbb{Z}_3$\textit{.  In} $D(-2, 1/9)$\textit{, if a parameter} $c$\textit{ is such that }
\[c\equiv -2 \mbox{ (mod } 3^k) \qquad \mbox{ but } \qquad c\not\equiv -2 \mbox{ (mod }3^{k+1}),\]
\textit{then 0 has orbit type }$(2,3^i)$\textit{ (mod }$3^{k+i}$\textit{) for }$i\geq 1$\textit{, for the corresponding map }$f_c(z) = z^2 + c$.

\begin{remark}
We expect a similar statement for parameters in $D(1,1/9)$ that are equivalent to the PCF parameter corresponding to critical orbit type $(2,3)$ up to a certain level.  This will be included in future work.
\end{remark}

The remainder of this section is devoted to proving Theorem \ref{c-2}.  The proof is calculation-heavy at times, so for ease of reading some of the computations have been omitted.  The reader is encouraged to verify any statements that lack full justification.
 
\begin{proof}[Proof of Theorem \ref{c-2}]
Consider the PCF parameter $c=-2$.  All other parameters in $D(-2,1/9)$ have infinite critical orbit, with orbit type of the form $(2,3^i)$ at every level in the $\mathbb{Z}_3$ tree.  The proof is based on the local structure of the iterates of $f_c$ for parameters near $c=-2$.  To begin, we show that a $\mathbb{Z}_3$ parameter 1 combinatorial distance away from the $-2$-branch in $\mathcal{M}_3$ corresponds to a map with critical orbit type $(2,3)$.

\begin{lemma}
Suppose $c\equiv -2$ (mod $3^k$) but $c\not\equiv -2$ (mod $3^{k+1})$ for $k\geq 2$, $c\in\mathbb{Z}_3$.  Then 0 has orbit type $(2,3)$ (mod $3^{k+1})$.
\end{lemma}

\begin{proof}
For such a $c\in\mathbb{Z}_3$ in $D(-2,1/9)$, the corresponding critical orbit is not fixed as orbit type $(2,1)$ (mod $3^{k+1}$).  We show that the first iterate of $f_c$ is locally a translation at the $k^{th}$ level.  This implies that the length of the critical orbit cycle will multiply by 3 at the ${k+1}^{st}$ level, resulting in critical orbit type $(2,3)$ (mod $3^{k+1})$.

Consider the local map of $f_c(z)=z^2+c$ for  $c\equiv -2$ (mod $3^k$), $c\not\equiv -2$ (mod $3^{k+1})$, centered at  $f_c^2(0)=c^2+c$, which is fixed (mod $3^k$).  Note that $c= -2+  l\cdot 3^k+\displaystyle\sum^{\infty}_{j=1} a_j 3^{k+j}$ for $l\in\{1,2\}$, $a_j \in\{ 0,1,2\}$.  

\[
\begin{array}{rcl}
\frac{1}{3^k}(f_c(c^2+c+3^kz) - (c^2+c)) & = & \frac{1}{3^k}((c^2+c+3^kz)^2+c)-(c^2+c))\\
 & = & \frac{1}{3^k}((c^2+c)^2+2(c^2+c)3^k z + 3^{2k}z^2+c-(c^2+c))\\
 & = & \frac{1}{3^k}(c^4+2c^3+c^2+2(c^2+c)3^k z + 3^{2k}z^2-c^2)\\
 & = & \frac{1}{3^k}(c^4+2c^3+2(c^2+c)3^k z + 3^{2k}z^2)\\
 & = & \frac{1}{3^k}(c^4+2c^3)  + 2(c^2+c)z + 3^kz^2\\
 & = & \frac{1}{3^k}(-8l 3^k + O(3^{k+1}))  + (4+O(3^{k+1}))z + 3^kz^2\\ 
& = & -8l +O(3)  + (4+O(3^{k+1}))z + 3^kz^2\\ 
& \equiv & z+ b \qquad \mbox{(mod }3),

\end{array}
\]
where $b\equiv -8l $ (mod $3$) $\not\equiv 0$ (mod $3$) since $l\in\{1,2\}$.

\end{proof}

At this point we call upon linearization to complete the picture.  For parameters 1 combinatorial distance away from the $-2$-branch at the ${k+1}^{st}$ level, $f_c$ has critical orbit type $(2,3)$ (mod $3^{k+1})$, for any $k\geq 2$.  Such $f_c$ have a fixed point $x=(1+\sqrt{1-4c})/2 \equiv 2$ (mod $3^k$).  We show that a) there exists a disk of linearization around $x$ in which the iterates of $f_c$ behave as the linear map $\lambda z$, for $\lambda = (f_c)'(x)$; and b) $f_c^2(0)$ is close enough to $x$ that it is in the linearization disk containing $x$.  We also show that $\lambda z$ is locally a translation, so the cycle length of vertices in the disk other than $x$ will increase by a factor of 3 at each level, under iteration of $f_c$. 

In terms of existence, any such $f_c$ is locally linearizable around its fixed points by the Non-Archimedean Siegel Theorem of Herman and Yoccoz \cite{Herman}, since $|\lambda|=1$ and $\lambda$ is not a root of unity.  It remains to show approximately how big the linearization disks are, using the calculations outlined in \cite{Olof}.  We include relevant definitions in the proof of Lemma \ref{radius} below.

\begin{lemma}\label{radius}
The map $f_{-2}(z)=z^2-2$ has a disk of linearization of radius at least $3^{-3/2}$ around the fixed point $z=2$.  
\end{lemma}

\begin{proof}

For $z=2$, the multiplier $\lambda_2 = 2\cdot 2 = 4$ satisfies $0<|1-\lambda|\leq 1/p$, with $ | 1-\lambda_2|= |3|=1/3$.  We may not use the exact radius of the linearization disk $r(P_{\lambda})$ given in \cite{Olof} Theorem 5.1 since $1/p\nless |1-\lambda|< 1$, so we take $\tilde r(\lambda)$ to be a sufficient lower bound.
\[
\tilde r(\lambda) := R(s+1)^{\frac{1}{m}} p^{-\frac{s-t}{mp^s}} |1-\lambda^m|^{s\frac{p-1}{mp}}\cdot |\gamma_0 - \lambda^m|^{\frac{1}{mp^{s-t}}},
\]
where $R(s) = p^{-\frac{1}{p^{s-1}(p-1)}}$, and $s\geq 0$ is the integer for which $R(s)\leq |1-\lambda^m|<R(s+1)$, with $m =$min$\{n\in\mathbb{Z} : n\geq 1, |1-\lambda^n|<1\}$.  Also, $\gamma_0$ is the ramified ${p^s}$-root of unity that minimizes $|\gamma_0 - \lambda^m|$, and $t\geq 0$ is the integer such that $R(t) \leq | \gamma_0 - \lambda^m|<R(t+1)$.

In our case of $\lambda_2 = 4$ and $p = 3$, we have the following:  $m=1$ and $s=0$, since $s$ must satisfy $3^{-\frac{1}{3^{s-1}(2)}}\leq |1-\lambda_2|< 3^{-\frac{1}{3^s(2)}}$.  Then $\gamma_0=1$, and so $t=0$. 
\[
\tilde r(\lambda_2)= (3^{-\frac{1}{2}})^1\cdot 3^0 \cdot |1-\lambda_2^1|^0 \cdot |1-\lambda_2^1|^{\frac{1}{3^0}} = 3^{-\frac{1}{2}}\cdot |1-\lambda_2| = 3^{-\frac{1}{2}}\cdot 3^{-1} = 3^{-\frac{3}{2}}.
\]

\end{proof}

\begin{corollary}
The map $f_c(z)=z^2+c$, with $c\in\mathbb{Z}_3$ such that $c\equiv -2$ (mod $3^k$) but $c\not\equiv -2$ (mod $3^{k+1})$, $k\geq 2$, has a disk of linearization of radius at least $3^{-3/2}$ around the fixed point $x=(1+\sqrt{1-4c})/2$.

\end{corollary}

\begin{proof}

Using the calculations in the proof of Lemma \ref{radius} above, we see that for $c\in\mathbb{Z}_3$ sufficiently close to -2, the multiplier of the corresponding fixed point, $x$, is close enough to $\lambda_2$ so that each piece in the definition of $\tilde r(\lambda_2)$ is the same as for $\tilde r(\lambda_x)$.  Thus, the disk of linearization around $x$ has radius at least $3^{-\frac{3}{2}}$.
\end{proof}

Now it remains to show that $|f_c^2(0)-x|\leq 3^{-\frac{3}{2}}$.  To do this, it will not be taken for granted that $x$ is an element of $\mathbb{Z}_3$.  It is an algebraic number, however, so it must be in some finite extension of $\mathbb{Q}_3$.  We may write  $\displaystyle x= 2 + h\cdot \pi^d + \sum_{i=1}^{\infty} b_i \pi^ {d+i}$ where $h\neq0$, and $3^k\leq \pi^d<3^{k+1}$, with $|\pi| = 3^{-\frac{a}{b}}$. 

We first fix $k=2$ and show that $9|(f_c^2(0)-x)$, so $|f_c^2(0)-x|\leq 3^{-2}<3^{-\frac{3}{2}}$.  Note that in this case $\displaystyle c=-2+l\cdot 3^2 + \sum_{j=1}^\infty a_j3^{2+j}$, with $l=1,2$.

\[
\begin{array}{rcl}
|c^2+c-x| & = & | 4 - 8l3^2 + l^23^4 - 4\sum_{j=1}^\infty a_j3^{2+j}+2l3^2\sum_{j=1}^\infty a_j3^{2+j} + (\sum_{j=1}^\infty a_j3^{2+j})^2\\
&& + -2 + l3^2+ \sum_{j=1}^\infty a_j3^{2+j} - (2 + h\pi^d + \sum_{i=1}^{\infty} b_i \pi^ {d+i})|\\
&=& |-7l3^2 + l^23^4 + (-3 + 2l3^2) (\sum_{j=1}^\infty a_j3^{2+j}) + (\sum_{j=1}^\infty a_j3^{2+j})^2 - h\pi^d \\
&& - \sum_{i=1}^{\infty} b_i \pi^ {d+i}|\\
& = & |3^2| |-7l + l^23^2 + (-3 + 2l3^2) (\sum_{j=1}^\infty a_j3^{j}) +  (\sum_{j=1}^\infty a_j3^{j})(\sum_{j=1}^\infty a_j3^{2+j})\\
&& - h\frac{\pi^d}{3^2} - \sum_{i=1}^{\infty} b_i \frac{\pi^ {d+i}}{3^2}|\\
&=& 3^{-2} \cdot 1
\end{array}
\] 
Since $\pi = 3^{\frac{a}{b}}$, with $3^2\leq \pi^d < 3^3$, we may safely factor out $3^2$ and have a power of $\pi$ leftover:  $\frac{\pi^d}{3^2} = \pi^e$ where $1\leq \pi^e < 3$.  Further, since $l= 1$ or 2, $3\nmid7l$ so $|c^2+c-x|/|3^2|=1$.  

For $k>2$, we see that $|c^2+c-x|/|3^2|\leq1$, since the remainder may be divisible by some higher power of 3.  This confirms that for any $k\geq 2$,  $|c^2+c-x|\leq3^{-2}<3^{-\frac{3}{2}}$.  The second iterate of 0 is therefore close enough to the fixed point $x$ to be in the disk of linearization around $x$.

We now take a closer look at the dynamics of $\lambda z$ in the linearization disk.  The multiplier $\lambda =2x$, so we may write $\displaystyle \lambda= 4 + 2h\cdot \pi^d + 2\sum_{i=1}^{\infty} b_i \pi^ {d+i}$, where $h\neq0$ and $3^k\leq \pi^d<3^{k+1}$.  We show that the iterates of $\lambda z$ behave locally as nontrivial translations, away from the fixed point.  

\begin{lemma}\label{linear}
Suppose $\lambda \equiv 4 $ (mod $3^k$) but $\lambda \not\equiv 4 $ (mod $3^{k+1}$), so $\lambda = 4 + 2h\cdot \pi^d + 2\sum_{i=1}^{\infty} b_i\pi^{d+i}$ for $h\in\{1,2\}$, $b_i\in\{0,1,2\}$, $3^k\leq \pi^d < 3^{k+1}$.  Then combinatorial distance $n$ away from the 0-branch, the map $g(z) = \lambda z$ behaves locally as a nontrivial translation, $z+1$ or $z+2$.
\end{lemma}

\begin{proof}

We prove this using induction, considering the local map centered at a point $v$ that is combinatorial distance $n$ away from the 0-branch.  Then $v \equiv 0$ (mod $3^j$) but $v \not\equiv 0$ (mod $3^{j+n}$) for $n\in\mathbb{N}$, so $v \equiv u\cdot 3^j + \sum^{n-1}_{i=1}a_i3^{j+i}$ (mod $3^{j+n}$), with $u\in\{1,2\}$ and $a_i\in\{0,1,2\}$.  

For the base case $v$ is combinatorial distance 1 away from the 0-branch, so $v$ is fixed at the ${j+1}^{st}$ level.  We compute the local map:
\[
\begin{array}{rcl}
\frac{1}{3^{j+1}}(g(3^{j+1}z+v) - v) & = & \frac{1}{3^{j+1}}(\lambda\cdot 3^{j+1} z + \lambda v - v )\\
& = & \lambda z + \frac{\lambda - 1}{3^{j+1}}\cdot v\\
& \equiv & \lambda z + \frac{3(1+2h\cdot \frac{\pi^d}{3} + \cdots )}{3^{j+1}}(u \cdot 3^j) \qquad (\mbox{mod }3^{j+1})\\
& \equiv & \lambda z  + (1+2h\cdot \frac{\pi^d}{3} + \cdots )\cdot u \qquad (\mbox{mod }3^{j+1})\\
& \equiv & z + u \qquad (\mbox{mod } 3),
\end{array}
\] 
where $u\in\{1,2\}$.  Thus $g$ is locally a nontrivial translation, so the period of $v$ will increase by a factor 3 at the ${j+2}^{st}$ level.  We assume this is the case for all combinatorial distances away from 0 up to $n-1$, so that at each level the period of $v$ will increase by a factor of $3$.  Now we will show this is true for combinatorial distance $n$ away from 0, noting that $v$ will be in a $3^{n-1}$-cycle at the ${j+n}^{th}$ level.  We compute the local map:
\begin{align}
\frac{1}{3^{j+n}}(g^{3^{n-1}}(3^{j+n}z+v) - v) & =  \frac{1}{3^{j+n}}(\lambda^{3^{n-1}}\cdot 3^{j+n} z + \lambda^{3^{n-1}} v - v )\nonumber\\
& =  \lambda^{3^{n-1}} z + \frac{\lambda^{3^{n-1}} - 1}{3^{j+n}}\cdot v\nonumber\\
& =  \lambda^{3^{n-1}} z + \frac{b}{3^{j}}\cdot v \qquad (3 \nmid b)\label{eqn:1}\\
& \equiv  z + \overline{b\cdot i} \qquad (\mbox{mod } 3)\label{eqn:2},
\end{align}

where $\overline{b\cdot i} \in\{ 1,2\}$.  (\ref{eqn:1}) and (\ref{eqn:2}) follow from claims 1 and 2 below.

\subsection*{Claim 1:} \textbf{ $\lambda^{3^{n-1}}\equiv 1$ (mod $3^n$) for all $n$}

\noindent We prove this claim by induction:

Base case, $n=1$:  $\lambda^1 \equiv 4$ (mod $3^2$) $\equiv 1 $ (mod $3^1$).

Now assume that $\lambda^{3^{n-2}}\equiv 1$ (mod $3^{n-1}$), and we will show this holds for $\lambda^{3^{n-1}}$.  Note that $\lambda^{3^{n-2}} = 1+\displaystyle\sum^{\infty}_{i=0} e_i\pi^{d+i}$ with $e_i \in \{0,1,2\}$ and $3^k\leq \pi^d < 3^{k+1}$, so 
\[ 
\lambda^{3^{n-2}} \equiv 1\mbox{ (mod } 3^{n-1}) \equiv 1 + \displaystyle\sum^{N}_{i=l} e_i\pi^{d+i} \mbox{ (mod } 3^{n}),
\]
where $3^{n-1}\leq \pi^{d+N}< 3^{n}$ but $3^n\leq \pi^{d+N+1}$, and $l$ is the least $i$ such that $3^{n-1}\leq \pi^{d+i}<3^n$.  Then 
\[
\begin{array}{rcl}
\lambda^{3^{n-1}}  =  (\lambda^{3^{n-2}})^3 & \equiv & \left(1+\displaystyle\sum^{N}_{i=l} e_i\pi^{d+i}\right)^3 \qquad (\mbox{mod }3^{n})\\
& = & 1 + 3\displaystyle\sum^{N}_{i=l} e_i\pi^{d+i}+ 3\left(\displaystyle\sum^{N}_{i=l} e_i\pi^{d+i}\right)^2 + \left(\displaystyle\sum^{N}_{i=l} e_i\pi^{d+i}\right)^3  \qquad (\mbox{mod }3^{n})\\
& \equiv & 1 \qquad (\mbox{mod }3^{n}),\\
\end{array}
\]
since $3^n\leq 3\pi^{d+l}<3^{n+1}$.  Thus, $3^n | (\lambda^{3^{n-1}}-1)$, and we let the remainder $\frac{\lambda^{3^{n-1}}-1}{3^n}=b$.

\subsection*{Claim 2:} \textbf{ $\lambda^{3^{n-1}}\not\equiv 1$ (mod $3^{n+1}$) for all $n$}

\noindent We prove this claim by induction:

Base case, $n=1$: $\lambda^1 \equiv 4$ (mod $3^2$) $\not\equiv 1 $ (mod $3^2$).

Now assume that $\lambda^{3^{n-2}}\not\equiv 1$ (mod $3^{n}$), and we will show this holds for $\lambda^{3^{n-1}}$. Note that $\lambda^{3^{n-2}} = 1 + \displaystyle\sum^{N}_{i=l} e_i\pi^{d+i}$  (mod $3^{n}$) where at least one $e_i\neq 0$.  Let $e_m$ be the first non-zero $e_i$, noting that $l\leq m\leq N$.  So

\[
\lambda^{3^{n-2}} \equiv 1 + e_m\pi^{d+m} +  \displaystyle\sum^{N}_{i=m+1} e_i\pi^{d+i} \mbox{ (mod } 3^n)\equiv 1 + e_m\pi^{d+m} +  \displaystyle\sum^{M}_{i=m+1} e_i\pi^{d+i} \mbox{ (mod } 3^{n+1}),
\]
where $3^{n}\leq \pi^{d+M}< 3^{n+1}$ but $3^{n+1}\leq \pi^{d+M+1}$, and $e_m\neq 0$.  Then

\[
\begin{array}{rcl}
\lambda^{3^{n-1}}  =  (\lambda^{3^{n-2}})^3 & \equiv & \left(1 + e_m\pi^{d+m} +  \displaystyle\sum^{M}_{i=m+1} e_i\pi^{d+i} \right)^3 \qquad (\mbox{mod }3^{n+1})\\
& \equiv & 1 + 3e_m\pi^{d+m}  + 3(e_m\pi^{d+m} )^2 + (e_m\pi^{d+m} )^3\\
& & + 3(1 + 2e_m\pi^{d+m} + (e_m\pi^{d+m} )^2)\displaystyle\sum^{M}_{i=m+1} e_i\pi^{d+i}\\
& &  +3(1+e_m\pi^{d+m} )\left(\displaystyle\sum^{M}_{i=m+1} e_i\pi^{d+i}\right)^2 + \left( \displaystyle\sum^{M}_{i=m+1} e_i\pi^{d+i}\right)^3 \qquad (\mbox{mod }3^{n+1})\\
& \equiv & 1 + 3e_m\pi^{d+m} + 3 \displaystyle\sum^{N}_{i=m+1} e_i\pi^{d+i}\qquad (\mbox{mod }3^{n+1})\\
& \not\equiv & 1 \qquad (\mbox{mod }3^{n+1})\\
\end{array}
\]
since $e_m\neq 0$ and $3^n\leq 3\pi^{d+m}<3^{n+1}$.  Thus, $3^{n+1} \nmid (\lambda^{3^{n-1}}-1)$, so $b$ is not divisible by 3.  In particular,  $\overline{b\cdot i} \in\{ 1,2\}$.

\end{proof}

We now consider the implications of the linearization disk.  The dynamics of $f_c$ will match that of $\lambda z$ for all points inside the disk.  By Lemma \ref{linear}, the map $\lambda z$ behaves locally as a non-trivial translation.  Translation maps have local period equal to the residue characteristic of the base field, which is 3 in this case.  Any vertex $\zeta = D(a,3^{-n})\not\equiv x$ will be in a 3-cycle (mod $3^{n+1}$) under iteration of $\lambda z$, so that the total length of the cycle containing $\zeta$ will increase by a factor of 3, for all $n$.  In particular, the total length of the cycle containing $\zeta = f_c^2(0)$ calculated (mod $3^{n+i}$) will increase by a factor of 3 for all $i$, provided that it is contained in the linearization disk around the fixed point, $x$, for $f_c$.  We have shown that for $c\equiv -2 $ (mod $3^k$) but $c\not\equiv -2 $ (mod $3^{k+1}$), $f_c^2(0)$ is in the linearization disk around $x$ for $k\geq 2$.  Therefore the corresponding critical orbit will be of orbit type (2,1) (mod $3^k$), (2,3) (mod $3^{k+1}$), and in general, $(2, 3^i)$ (mod $3^{k+i}$), for all $i$, and any $k\geq 2$.

\end{proof}

\section{Dynamics in $\mathbb{Z}_2$}\label{p2}

To conclude this article, we briefly examine the case of $p=2$.  The even prime is often treated separately, a sort of special case that does not always follow the same patterns as the odd primes - and this project is no different.  Most of the results in this paper hold only for odd primes, and that is at least partially because some of the technical tools used, such as Hensel's lemma, do not apply  to $f_c$ when $p=2$.  In lieu of a proof of this fact, we offer a counterexample and a few straightforward calculations to describe the parameter space $\mathbb{Z}_2$ as we best understand it at this time.

\begin{theorem}
Theorem \ref{preper} and Proposition \ref{per} do not hold for $p=2$.
\end{theorem}

For example, consider the map $f(z) = z^2-2$ in $\mathbb{Z}_2$.  The critical point is fixed (mod 2), which would typically indicate the presence of a hyperbolic component of radius 1, in which every parameter other than $c=0$ corresponds to an infinite critical orbit.  However, modulo higher powers of 2, the critical orbit eventually resolves itself as type (2,1), which is where it remains over $\mathbb{Z}_2$.

The critical orbit of  $f(z) = z^2 -2$:
\[
\begin{array}{lcl}
 0 \circlearrowleft & (\mbox{mod } 2), & \mbox{ orbit type } (0,1)\\
  0 \rightarrow 2 \circlearrowleft & (\mbox{mod } 2^2), & \mbox{ orbit type } (1,1)\\
 0 \rightarrow 6 \rightarrow  2 \circlearrowleft & (\mbox{mod } 2^3), & \mbox{ orbit type } (2,1)\\
 0 \rightarrow 14 \rightarrow  2 \circlearrowleft & (\mbox{mod } 2^4), & \mbox{ orbit type } (2,1)\\
 & \vdots& \\
 0 \rightarrow -2 \rightarrow  2 \circlearrowleft & \mbox{in } \mathbb{Z}_2, & \mbox{ orbit type } (2,1)\\

\end{array}
\]

The fixed point $x=2$ in this example is an attracting fixed point, as opposed to an indifferent fixed point as when $c=-2$, $p\geq 3$.  The multiplier $\lambda = 2x = 4$ satisfies $|\lambda| = 1/4 < 1$.  In fact, every cycle which is contained in the disk $\{z\in \mathbb{C}_2 : |z|\leq 1\}$ for $f_c$ with $c\in\mathbb{Z}_2$  is attracting.  Further, the critical point is attracted to exactly one of these attracting cycles for each $f_c$ corresponding to $c\in\mathbb{Z}_2$.

Thus one fundamental difference between the $p=2$ case and $p\geq3$ is that the hyperbolic subset of the 2-adic Mandelbrot set is exactly the 2-adic Mandelbrot set (see the brief discussion in \cite{Jones2}).  The radius of each hyperbolic component is 1, and as Figure \ref{Z2} suggests, there is only one component associated to an attracting fixed point and one component associated to an attracting 2-cycle in $\mathbb{Z}_2$.

\begin{figure}[h]
\centerline{ \includegraphics[width=3in]{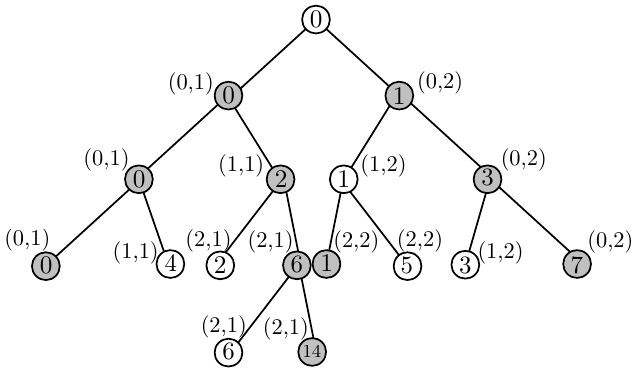}}
 \caption{$\mathbb{Z}_2$ as a parameter space, with the critical orbit type corresponding to each parameter, calculated at that level.  The 3 paths shaded in gray indicate the known PCF parameters: $c=0$, $c=-2$ and $c=-1$.  All other branches seem to correspond to parameters with infinite critical orbit.\protect\label{Z2}}
\end{figure}

Because Proposition \ref{fixed tail} does not apply in the $p=2$ setting, since neither of the (mod 2) critical orbits are pre-periodic, we cannot be sure that the 3 PCF parameters indicated in Figure \ref{Z2} are the only ones that live in $\mathbb{Z}_2$.  The critical orbit trees associated to these parameters are shown in Figure \ref{Z2 trees}; notice that the tail branches in the (2,1) tree, rather than the cycle.  It would be interesting to investigate whether there is a bound on the number of times the tail may branch for a finite critical orbit.  This information would allow us to find all of the PCF points in $\mathbb{Z}_2$.

\begin{figure}[h]
\centerline{ \includegraphics[width=1.7in]{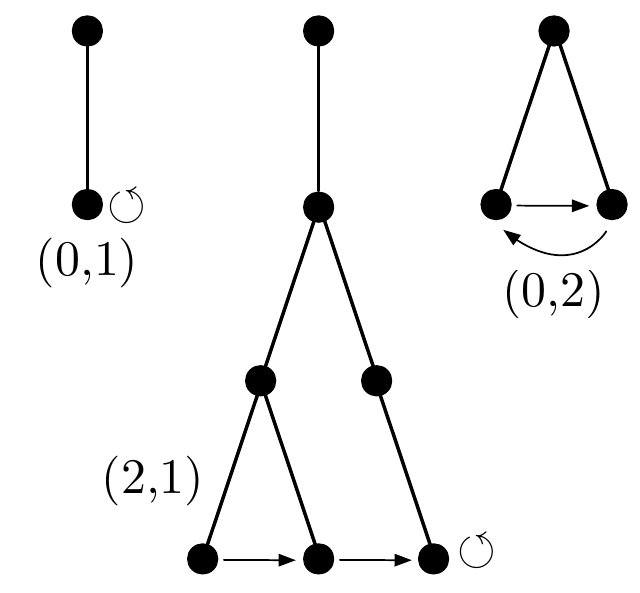}}
 \caption{The 3 known distinct finite critical orbit trees in $\mathbb{Z}_2$, with corresponding critical orbit type labeled below.\protect\label{Z2 trees}}
\end{figure}

This is certainly a topic that is wide open with the potential for future work.  And once all of the possible critical orbit structures are fully understood for quadratic polynomials over $\mathbb{Z}_p$, for all primes $p\geq 2$,  a natural next step will be to explore the possible structures in finite extensions of $\mathbb{Q}_p$, and finally in $\mathbb{C}_p$.  Only then will we have a complete picture of $p$-adic Hubbard trees for polynomials of the form $z^2 + c$.


\end{document}